\documentclass[12pt]{article}

\usepackage{amsmath, epsfig, cite, setspace}
\usepackage{amssymb}
\usepackage{amsfonts}
\usepackage{latexsym}
\usepackage{amsthm}

\makeatletter
\renewcommand{\@seccntformat}[1]{{\csname the#1\endcsname}.\hspace{.5em}}
\makeatother

\newtheorem{theorem}{Theorem}[section]

\newtheorem{corollary}[theorem]{Corollary}

\newtheorem{lemma}[theorem]{Lemma}

\renewcommand{\thefootnote}{*}

\numberwithin{equation}{section}

\begin{document}

\begin{center}
{\large\bf  Some $q$-supercongruences for multiple basic\\[1mm]
hypergeometric series}
\end{center}

\vskip 2mm \centerline{Chuanan Wei}
\begin{center}
{\footnotesize School of Biomedical Information and Engineering,\\
  Hainan Medical University, Haikou 571199, China\\
{\tt Email address:weichuanan78@163.com} }
\end{center}


\vskip 0.7cm \noindent{\bf Abstract.} In terms of  several summation
and transformation formulas for basic hypergeometric series, two
forms of the Chinese remainder theorem for coprime polynomials, the
creative microscoping method introduced by Guo and Zudilin, Guo and
Li's lemma, and El Bachraoui's lemma, we establish some
$q$-supercongruences for multiple basic hypergeometric series modulo
the fifth and sixth powers of a cyclotomic polynomial. In detail, we
generalize  Guo and Li's two $q$-supercongruences for double basic
hypergeometric series, which are related to $q$-analogues of Van
Hamme's (C.2) supercongruence and Long's supercongruence,
respectively. In addition, we also present two conclusions for
double and triple hypergeometric series associated with Van Hamme's
(D.2) supercongruence.



\renewcommand{\thefootnote}{**}

\section{Introduction}
For a complex number $x$ and a nonnegative integer $n$, define the
shifted-factorial to be
\[(x)_{n}=\Gamma(x+n)/\Gamma(x),\]
where $\Gamma(x)$ denotes the usual Gamma function.  All over the
paper, $p$ is an odd prime and $\Gamma_p(x)$ stands for the $p$-adic
Gamma function. In 1997, Van Hamme \cite[(C.2) and (D.2)]{Hamme}
proposed the following nice conjectures:
\begin{equation}\label{van-hamme-a}
\sum_{k=0}^{(p-1)/2}(4k+1)\frac{(\frac{1}{2})_k^4}{k!^4}\equiv
 p \pmod{p^3},
\end{equation}
 and for $p\equiv 1\pmod 6$,
\begin{equation}\label{van-hamme-b}
\sum_{k=0}^{(p-1)/3}(6k+1)\frac{(\frac{1}{3})_k^6}{k!^6}\equiv
 -p\Gamma_p(\tfrac{1}{3})^9 \pmod{p^4}.
\end{equation}
Some years later, Long and Ramakrishna \cite{LR-b} certified that
\eqref{van-hamme-b} is right modulo $p^6$.
  In 2011, Long \cite{LR-a} obtained the following beautiful supercongruence:
for $p>3$,
\begin{align}
\sum_{k=0}^{(p-1)/2}(4k+1)\frac{(\frac{1}{2})_k^6}{k!^6} &\equiv
p\sum_{k=0}^{(p-1)/2}\frac{(\frac{1}{2})_k^4}{k!^4}\pmod{p^4}.
\label{long}
\end{align}
  More supercongruences can be seen in the papers
\cite{GLS,He,Liu-b,Liu-c}.

For two complex numbers $x$ and $q$, define the $q$-shifted
factorial as
 \begin{equation*}
(x;q)_{0}=1\quad\text{and}\quad
(x;q)_n=(1-x)(1-xq)\cdots(1-xq^{n-1})\quad \text{when}\quad
n\in\mathbb{Z}^{+}.
 \end{equation*}
For simplicity, we frequently adopt the notation
\begin{equation*}
(x_1,x_2,\dots,x_m;q)_{n}=(x_1;q)_{n}(x_2;q)_{n}\cdots(x_m;q)_{n},
 \end{equation*}
where $m\in\mathbb{Z}^{+}$ and $n\in\mathbb{Z}^{+}\cup\{0\}.$
 Throughout the paper, let $[r]$ denote the $q$-integer $(1-q^r)/(1-q)$ and let $\Phi_n(q)$ stand for
the $n$-th cyclotomic polynomial in $q$:
\begin{equation*}
\Phi_n(q)=\prod_{\substack{1\leqslant k\leqslant n\\
\gcd(k,n)=1}}(q-\zeta^k),
\end{equation*}
where $\zeta$ is an $n$-th primitive root of unity. Recently, Guo
and Wang \cite{GW} established a $q$-analogue of
\eqref{van-hamme-a}: for any positive odd integer $n$,
\begin{equation*}
\sum_{k=0}^{(n-1)/2}[4k+1]\frac{(q;q^2)_k^4}{(q^2;q^2)_k^4} \equiv
[n]q^{(1-n)/2}+[n]^3q^{(1-n)/2}\frac{(n^2-1)(1-q)^2}{24}\pmod{[n]\Phi_n(q)^3}.
\end{equation*}
The author \cite{Wei-a} deduced a $q$-analogue of
\eqref{van-hamme-b}: for any positive integer $n\equiv 1\pmod 3$,
\begin{align*}
\sum_{k=0}^{(n-1)/3}[6k+1]\frac{(q;q^3)_k^6}{(q^3;q^3)_k^6}q^{3k}
&\equiv[n]
\frac{(q^2;q^3)_{(n-1)/3}^3}{(q^3;q^3)_{(n-1)/3}^3}
\\[1mm]
&\hspace{-30mm}\times\:
\bigg\{1+[n]^2(2-q^{n})\sum_{j=1}^{(n-1)/3}\bigg(\frac{q^{3j-1}}{[3j-1]^2}-\frac{q^{3j}}{[3j]^2}\bigg)\bigg\}
\pmod{[n]\Phi_n(q)^4},
\end{align*}
and the author \cite{Wei} found a $q$-analogue of \eqref{long}: for
any positive odd integer $n$,
\begin{align*}
\sum_{k=0}^{(n-1)/2}[4k+1]\frac{(q;q^2)_k^6}{(q^2;q^2)_k^6}q^k
&\equiv[n]q^{(1-n)/2}\bigg\{1+[n]^2\frac{(n^2-1)(1-q)^2}{24}\bigg\}\\[5pt]
&\quad\times\sum_{k=0}^{(n-1)/2}\frac{(q;q^2)_k^4}{(q^2;q^2)_k^4}q^{2k}\pmod{[n]\Phi_n(q)^3}.
\end{align*}

In 2021, El Bachraoui \cite{Mohamed}  proved several
$q$-supercongruences for double basic hypergeometric series.
 Motivated by El Bachraoui's work, Guo and Li \cite[Theorems 1.4 and 5.4]{GuoLi}
gave the following two $q$-supercongruences for double basic
hypergeometric series related to \eqref{van-hamme-a} and
\eqref{long}: for any positive odd integer $n$,
\begin{equation}
\sum_{i+j\leq n-1}c_q(i)c_q(j)\equiv [n]^2q^{1-n}
\pmod{[n]\Phi_n(q)^3}, \label{eq:guo-li-a}
\end{equation}
where
$$
 c_q(k)=[4k+1]\frac{(q;q^2)_k^4}{(q^2;q^2)_k^4}
\quad\text{with}\quad k\geq 0,
$$
and for any positive odd integer $n$,
\begin{equation}
\sum_{i+j\leq n-1}c_q(i)c_q(j)\equiv
[n]^2q^{1-n}\bigg\{\sum_{k=0}^{(n-1)/2}\frac{(q;q^2)_k^4}{(q^2;q^2)_k^4}q^{2k}\bigg\}^2
\pmod{[n]\Phi_n(q)^2}, \label{eq:guo-li-b}
\end{equation}
where
 $$
  c_q(k)=[4k+1]\frac{(q;q^2)_k^6}{(q^2;q^2)_k^6}q^k\quad\text{with}\quad k\geq
  0.
 $$
Song and Wang \cite[Corollary 2.3]{SW} hoped to furnish a
generalization of \eqref{eq:guo-li-a} modulo $[n]\Phi_n(q)^4$.
However, their derivation from Theorem 2.2 to Corollary 2.3 by
taking $d=q$ is wrong since that the factor $(1-q^n)$ appears in the
denominator of the series on the left-hand side. A generalization of
\eqref{eq:guo-li-b} due to Song and Wang \cite[Corollary 2.5]{SW}
reads
\begin{align}
\sum_{i+j\leq n-1}c_q(i)c_q(j)&\equiv
[n]^2q^{1-n}\sum_{k=0}^{(n-1)/2}\frac{(q;q^2)_k^4}{(q^2;q^2)_k^4}q^{2k}
\notag\\[1mm]
&\quad\times\sum_{k=0}^{(n-1)/2}\frac{(q;q^2)_k^4}{(q^2;q^2)_k^4}q^{2k}
\bigg\{1-2[n]^2\sum_{t=1}^{2k}\frac{q^t}{[t]^2}\bigg\}
\pmod{[n]\Phi_n(q)^4}. \label{eq:Song-Wang}
\end{align}
For more results on $q$-supercongruences, the reader is referred to
the papers \cite{GuoZu-c,Li,Wei-c,XW}. Encouraged by these works, we
shall establish the generalizations of \eqref{eq:guo-li-a} and
\eqref{eq:Song-Wang} in the following two theorems.

\begin{theorem}\label{thm-a}
Let $n$ be a positive odd integer and
$$\theta_q(k)=[4k+1]\frac{(q;q^2)_k^4}{(q^2;q^2)_k^4}\quad\text{with}\quad k\geq
  0.$$
 Then, modulo $[n]\Phi_n(q)^4$,
\begin{align}
\sum_{i+j\leq n-1}\theta_q(i)\theta_q(j)&\equiv
[n]^2q^{1-n}-2[n]^4\sum_{t=1}^{(n-1)/2}\frac{q^{2t+1}}{[2t]^2}.
\label{eq:wei-a}
\end{align}
\end{theorem}

Choosing $n=p^s$ and taking $q\to 1$ in Theorem \ref{thm-a}, we get
the supercongruence.

\begin{corollary}\label{cor-a}
Let $p>3$ be an odd prime and $s$ a positive integer.  Then for
$i,j\geq0$,
\begin{align*}
&\sum_{i+j\leq
p^s-1}(4i+1)(4j+1)\frac{(\frac{1}{2})_i^4(\frac{1}{2})_j^4}{(1)_i^4(1)_j^4}
\equiv p^{2s}-\frac{p^{4s}}{2}H_{(p^s-1)/2}^{(2)}\pmod{p^{s+4}},
\end{align*}
where the harmonic numbers of $2$-order have been defined by
\[H_{m}^{(2)}
  =\sum_{i=1}^m\frac{1}{i^{2}}.\]
\end{corollary}

\begin{theorem}\label{thm-b}
Let $n$ be a positive odd integer and
$$\theta_q(k)=[4k+1]\frac{(q;q^2)_k^6}{(q^2;q^2)_k^6}q^k\quad\text{with}\quad k\geq
  0.$$ Then,
modulo $[n]\Phi_n(q)^5$,
\begin{align}
\sum_{i+j\leq n-1}\theta_q(i)\theta_q(j)&\equiv
[n]^2q^{1-n}\sum_{k=0}^{(n-1)/2}\frac{(q;q^2)_k^4}{(q^2;q^2)_k^4}q^{2k}
\notag\\[1mm]
&\quad\times\sum_{k=0}^{(n-1)/2}\frac{(q;q^2)_k^4}{(q^2;q^2)_k^4}q^{2k}
\bigg\{1-2(2-q^n)[n]^2\sum_{t=1}^{2k}\frac{q^t}{[t]^2}\bigg\}.
\label{eq:wei-b}
\end{align}
\end{theorem}

Fixing $n=p^s$ and taking $q\to 1$ in Theorem \ref{thm-b}, we have
the result.

\begin{corollary}\label{cor-b}
Let $p$ be an odd prime and $s$ a positive integer.  Then for
$i,j\geq0$,
\begin{align*}
&\sum_{i+j\leq
p^s-1}(4i+1)(4j+1)\frac{(\frac{1}{2})_i^6(\frac{1}{2})_j^6}{(1)_i^6(1)_j^6}
\\[1mm]
&\quad\equiv
p^{2s}\sum_{k=0}^{(p^s-1)/2}\frac{(\frac{1}{2})_k^4}{(1)_k^4}\sum_{k=0}^{(p^s-1)/2}\frac{(\frac{1}{2})_k^4}{(1)_k^4}
\bigg\{1-2p^{2s}H_{2k}^{(2)}\bigg\}\pmod{p^{s+5}}.
\end{align*}
\end{corollary}

Song and Wang \cite{SW} conjectured that \eqref{eq:wei-b} is true
modulo $[n]^2\Phi_n(q)^4$. We verify partially their conjecture in
Theorem \ref{thm-b}. Furthermore, we shall also establish the
following $q$-supercongruence for multiple basic hypergeometric
series.

\begin{theorem}\label{thm-c}
Let $d,m,n,r$ be positive integers subject to $n\equiv r\pmod{d}$,
$(r,d)=1$, and $d\geq m\geq2$. Let $\theta_q(k)$ represent
$$\theta_q(k)=[2dk+r]\frac{(q^r;q^d)_k^6}{(q^d;q^d)_k^6}q^{(2d-3r)k}\quad\text{with}\quad k\geq
  0.$$
 Then, modulo $\Phi_n(q)^6$,
\begin{align}
&\sum_{i_1+i_2+\cdots+i_m\leq n-1} \theta_q(i_1)\cdots\theta_q(i_m)
\notag\\[1mm]
&\quad\equiv
[n]^mq^{mr(r-n)/d}\frac{(q^{2r};q^d)_{(n-r)/d}^m}{(q^{d};q^d)_{(n-r)/d}^m}
\bigg\{\sum_{k=0}^{(n-r)/d}\frac{(q^r;q^d)_k^3(q^{d-r};q^d)_k}{(q^d;q^d)_k^3(q^{2r};q^d)_k}q^{dk}\bigg\}^{m-1}
\notag\\[1mm]
&\qquad\times\sum_{k=0}^{(n-r)/d}\frac{(q^r;q^d)_k^3(q^{d-r};q^d)_k}{(q^d;q^d)_k^3(q^{2r};q^d)_k}q^{dk}
\bigg\{1-m(2-q^n)[n]^2\sum_{t=1}^{k}\bigg(\frac{q^{dt-d+r}}{[dt-d+r]^2}+\frac{q^{dt}}{[dt]^2}\bigg)\bigg\}.
\label{eq:wei-c}
\end{align}
\end{theorem}

Setting $n=p$ and taking $q\to 1$ in Theorem \ref{thm-c}, we catch
hold of the formula.

\begin{corollary}\label{cor-c}
Let $p$ be an odd prime and $d, m, r$ positive integers subject to
$p\equiv r\pmod{d}$, $(r,d)=1$ and $d\geq m\geq 2$. Then for
$i_1,i_2,\ldots,i_m\geq0$,
\begin{align*}
&\sum_{i_1+i_2+\cdots+i_m\leq p-1}
(2di_1+r)\cdots(2di_m+r)\frac{(\frac{r}{d})_{i_1}^6\cdots(\frac{r}{d})_{i_m}^6}{(1)_{i_1}^6\cdots(1)_{i_m}^6}
\\[1mm]
&\quad\equiv
p^m\frac{(\frac{2r}{d})_{(p-r)/d}^m}{(1)_{(p-r)/d}^m}\bigg\{\sum_{k=0}^{(p-r)/d}\frac{(\frac{r}{d})_k^3(\frac{d-r}{d})_k}{(1)_k^3(\frac{2r}{d})_k}\bigg\}^{m-1}
\\[1mm]
&\qquad
\times\sum_{k=0}^{(p-r)/d}\frac{(\frac{r}{d})_k^3(\frac{d-r}{d})_k}{(1)_k^3(\frac{2r}{d})_k}
\bigg\{1-mp^{2}\sum_{t=1}^{k}\bigg(\frac{1}{(dt-d+r)^2}+\frac{1}{(dt)^2}\bigg)\bigg\}\pmod{p^{6}}.
\end{align*}
\end{corollary}

The $(d,m,r)=(3,2,1)$ and  $(d,m,r)=(3,3,1)$  cases of Corollary
\ref{cor-c} are the following conclusions associated with
\eqref{van-hamme-b}.

\begin{corollary}\label{cor-d}
Let $p\equiv 1\pmod{6}$ be a prime.  Then for $i,j\geq0$,
\begin{align*}
&\sum_{i+j\leq p-1}
(6i+1)(6j+1)\frac{(\frac{1}{3})_i^6(\frac{1}{3})_j^6}{(1)_i^6(1)_j^6}
\\[1mm]
&\quad \equiv
p^{2}\frac{(\frac{2}{3})_{(p-1)/3}^2}{(1)_{(p-1)/3}^2}\sum_{k=0}^{(p-1)/3}\frac{(\frac{1}{3})_k^3}{(1)_k^3}
\\[1mm]
&\qquad\times\sum_{k=0}^{(p-1)/3}\frac{(\frac{1}{3})_k^3}{(1)_k^3}
\bigg\{1-2p^{2}\sum_{t=1}^{k}\bigg(\frac{1}{(3t-2)^2}+\frac{1}{(3t)^2}\bigg)\bigg\}\pmod{p^{6}}.
\end{align*}
\end{corollary}

\begin{corollary}\label{cor-e}
Let $p\equiv 1\pmod{6}$ be a prime.  Then for $i,j,k\geq0$,
\begin{align*}
&\sum_{i+j+k\leq p-1}
(6i+1)(6j+1)(6k+1)\frac{(\frac{1}{3})_i^6(\frac{1}{3})_j^6(\frac{1}{3})_k^6}{(1)_i^6(1)_j^6(1)_k^6}
\end{align*}
\begin{align*}
&\quad \equiv
p^{3}\frac{(\frac{2}{3})_{(p-1)/3}^3}{(1)_{(p-1)/3}^3}\bigg\{\sum_{k=0}^{(p-1)/3}\frac{(\frac{1}{3})_k^3}{(1)_k^3}\bigg\}^2
\\[1mm]
&\qquad\times\sum_{k=0}^{(p-1)/3}\frac{(\frac{1}{3})_k^3}{(1)_k^3}
\bigg\{1-3p^{2}\sum_{t=1}^{k}\bigg(\frac{1}{(3t-2)^2}+\frac{1}{(3t)^2}\bigg)\bigg\}\pmod{p^{6}}.
\end{align*}
\end{corollary}

The rest of the paper is arranged as follows. By means of several
summation and transformation formulas for basic hypergeometric
series, two forms of the Chinese remainder theorem for coprime
polynomials, the creative microscoping method introduced by Guo and
Zudilin
 \cite{GuoZu-b}, Guo and Li's lemma,
and El Bachraoui's lemma, we shall prove Theorem \ref{thm-a} in the
next section. Similarly, the proof of Theorems \ref{thm-b} and
\ref{thm-c} will be displayed in Sections 3 and 4, respectively.

\section{Proof of Theorem \ref{thm-a}}

For proving Theorem \ref{thm-a}, we require the following two
lemmas, where the former is from Guo and Li \cite{GuoLi} and the
latter can be viewed in El Bachraoui \cite{Mohamed-a} (see also
\cite{WY-a}).

\begin{lemma}\label{lemma-a}
Let $n$ be a positive odd integer. Let $\{a_0,a_1,\ldots,a_{n-1}\}$
be a sequence of numbers such that $a_k\equiv
-a_{(n-1)/2-k}\pmod{\Phi_n(q)}$ for $0\leq k \leq(n-1)/2$ and
$a_k\equiv -a_{(3n-1)/2-k}\pmod{\Phi_n(q)}$ for $(n+1)/2\leq k \leq
n-1$. Then
$$
\sum_{i+j\leq n-1}a_ia_j \equiv0\pmod{\Phi_n(q)}.
$$
\end{lemma}

\begin{lemma}\label{lemma-b}
Let $d,m,n$ be positive integers subject to $d\geq m\geq2$ and
$n\equiv1\pmod{d}$. Let $\{\lambda(k)\}_{k=0}^\infty$ be a complex
sequence. If $\lambda(k)=0$ for $(n+d-1)/d\leq k\leq n-1$, then
$$
\sum_{k_1+k_2+\cdots+k_m\leq
n-1}\lambda(k_1)\lambda(k_2)\cdots\lambda(k_m)
=\Bigg\{\sum_{k=0}^{(n-1)/d}\lambda(k)\Bigg\}^m.
$$
Furthermore, if $\lambda(ln+k)/\lambda(ln)=\lambda(k)$ for all
nonnegative integers $k$ and $l$ with $0\leqslant k\leqslant n-1$,
then
  \begin{align*}
&\sum_{k_1+k_2+\cdots+k_m=
ln+k}\lambda(k_1)\lambda(k_2)\cdots\lambda(k_m)=\sum_{k_1+k_2+\cdots+k_m=
k}\lambda(k_1)\lambda(k_2)\cdots\lambda(k_m)
\\[1mm]
&\qquad\:\:\times\sum_{s_1=0}^{l}\lambda(s_1n)\sum_{s_2=0}^{l-s_1}\lambda(s_2n)\cdots\sum_{s_{m-1}=0}^{l-s_1-\cdots-s_{m-2}}\lambda(s_{m-1}n)
\lambda((l-s_1-\cdots-s_{m-1})n).
\end{align*}
\end{lemma}

Now we are ready to give a  parametric generalization of Theorem
\ref{thm-a}.

\begin{theorem}\label{thm-d}
Let $n$ be a positive odd integer and
$$\lambda_q(k)=[4k+1]\frac{(q;q^2)_k^2(aq,q/a;q^2)_k}{(q^2;q^2)_k^2(q^2/a,aq^2;q^2)_k}\quad\text{with}\quad k\geq
  0.$$
Then, modulo $[n]\Phi_n(q)^2(1-aq^n)(a-q^n)$,
\begin{align}
\sum_{i+j\leq n-1}\lambda_q(i)\lambda_q(j) &\equiv
[n]^2q^{1-n}+[n]^2q^{1-n}\frac{(1-aq^n)(a-q^n)}{(1-a)^2}
\notag\\[1mm]
&\quad\times\bigg\{1-\frac{(q^2;q^2)_{(n-1)/2}^4}{(aq^2,q^2/a;q^2)_{(n-1)/2}^2}\bigg\}.
\label{eq:wei-d}
\end{align}
\end{theorem}

\begin{proof}
It is evident that the $n=1$ case of \eqref{eq:wei-d} is true. We
now assume that $n>1$. Set
$$\beta_q(k)=[4k+1]\frac{(aq,q/a,bq,q/b;q^2)_k}{(q^2/a,aq^2,q^2/b,bq^2;q^2)_k}.$$
A $q$-congruence due to Guo and Schlosser \cite[Lemma 3.1]{GS} can
be stated  as
\begin{equation}
\frac{(aq;q^2)_{(n-1)/2-k}}{(q^2/a;q^2)_{(n-1)/2-k}}\equiv
(-a)^{(n-1)/2-2k}\frac{(aq;q^2)_{k}}{(q^2/a;q^2)_{k}}q^{(n-1)^2/4+k}\pmod{\Phi_n(q)},
  \label{eq:wei-e}
\end{equation}
where $0\leq k\leq (n-1)/2$. Via \eqref{eq:wei-e}, it is not
difficult to understand that
\begin{equation}
\beta_q(k)\equiv -\beta_q((n-1)/2-k)\pmod{\Phi_n(q)}.
  \label{eq:wei-f}
\end{equation}
For $(n+1)/2\leq k\leq n-1$, we have
\begin{equation}
\frac{(aq;q^2)_{(3n-1)/2-k}}{(q^2/a;q^2)_{(3n-1)/2-k}}\equiv
(-a)^{(3n-1)/2-2k}\frac{(aq;q^2)_{k}}{(q^2/a;q^2)_{k}}q^{(3n-1)^2/4+k}\pmod{\Phi_n(q)}.
  \label{eq:wei-g}
\end{equation}
Through \eqref{eq:wei-g}, it is ordinary to discover that
\begin{equation}
\beta_q(k)\equiv -\beta_q((3n-1)/2-k)\pmod{\Phi_n(q)}.
  \label{eq:wei-h}
\end{equation}
According to \eqref{eq:wei-f}, \eqref{eq:wei-h}, and Lemma
\ref{lemma-a}, we obtain
\begin{align}
\sum_{i+j\leq n-1}\beta_q(i)\beta_q(j) &\equiv0\pmod{\Phi_n(q)}.
 \label{eq:wei-i}
\end{align}

Let $\zeta\neq 1$ be an $n$-th root of unity, not necessarily
primitive. Namely, $\zeta$ is a primitive $u$-th root of unity with
$u\mid n$. Thanks to the $n=u$ case of \eqref{eq:wei-i}, there is
$$
\sum_{i+j\leq u-1}\beta_\zeta(i)\beta_\zeta(j)=0.
$$
For all integers $k$ and $l$ satisfying $0\leqslant k\leqslant u-1$,
it is easy to realize that
$$\frac{\beta_\zeta(lu+k)}{\beta_\zeta(lu)}=\beta_\zeta(k).$$
With the help of the $d=m=2$ case of Lemma \ref{lemma-b}, we get
\begin{align*}
\sum_{i+j\leq n-1}\beta_\zeta(i)\beta_\zeta(j)
&=\sum_{t=0}^{n-1}\sum_{j=0}^{t}\beta_\zeta(j)\beta_\zeta(t-j)
\\[1mm]
&=\sum_{l=0}^{n/u-1}\sum_{k=0}^{u-1}\sum_{j=0}^{lu+k}\beta_\zeta(j)\beta_\zeta(lu+k-j)
\\[1mm]
&=\sum_{l=0}^{n/u-1}\sum_{k=0}^{u-1}\sum_{i=0}^{l}\beta_\zeta(iu)\beta_\zeta((l-i)u)\sum_{j=0}^{k}\beta_\zeta(j)\beta_\zeta(k-j)
\\[1mm]
&=\sum_{l=0}^{n/u-1}\sum_{i=0}^{l}\beta_\zeta(iu)\beta_\zeta((l-i)u)\sum_{k=0}^{u-1}\sum_{j=0}^{k}\beta_\zeta(j)\beta_\zeta(k-j)
\\[1mm]
&=0.
\end{align*}
Because the upper equation is correct for any $n$-th root of unit
$\zeta\neq 1$, we conclude that
\begin{equation}\label{eq:wei-j}
\sum_{i+j\leq n-1}\beta_q(i)\beta_q(j)\equiv 0\pmod{[n]}.
\end{equation}

Following Gasper and Rahman \cite{Gasper}, define the basic
hypergeometric series by
$$
_{r}\phi_{s}\left[\begin{array}{c}
a_1,a_2,\ldots,a_{r}\\
b_1,b_2,\ldots,b_{s}
\end{array};q,\, z
\right] =\sum_{k=0}^{\infty}\frac{(a_1,a_2,\ldots, a_{r};q)_k}
{(q,b_1,b_2,\ldots,b_{s};q)_k}\bigg\{(-1)^kq^{\binom{k}{2}}\bigg\}^{1+s-r}z^k.
$$
Then Watson's $_8\phi_7$ transformation (cf. \cite[Appendix
(III.18)]{Gasper}) can be expressed as
\begin{align}
& _{8}\phi_{7}\!\left[\begin{array}{cccccccc}
a,& qa^{\frac{1}{2}},& -qa^{\frac{1}{2}}, & b,    & c,    & d,    & e,    & q^{-m} \\
  & a^{\frac{1}{2}}, & -a^{\frac{1}{2}},  & aq/b, & aq/c, & aq/d, & aq/e, & aq^{m+1}
\end{array};q,\, \frac{a^2q^{m+2}}{bcde}
\right] \notag\\[5pt]
&\quad =\frac{(aq, aq/de;q)_{m}} {(aq/d, aq/e;q)_{m}}
\,{}_{4}\phi_{3}\!\left[\begin{array}{c}
aq/bc,\ d,\ e,\ q^{-m} \\
aq/b,\, aq/c,\, deq^{-m}/a
\end{array};q,\, q
\right]. \label{eq:watson}
\end{align}
Performing the replacements $q\to q^2,a\to q, b\to q^2, c\to q/b,
d\to bq, e\to q^{1+n},m\to (n-1)/2$ in \eqref{eq:watson}, we arrive
at
\begin{equation*}
\sum_{k=0}^{(n-1)/2}\tilde{\beta}_q(k)=[n](bq)^{(1-n)/2}\frac{(bq^2;q^2)_{(n-1)/2}}{(q^2/b;q^2)_{(n-1)/2}}
\sum_{k=0}^{(n-1)/2}\frac{(b,bq,q^{1+n},q^{1-n};q^2)_k}{(q,q^2,bq^2,bq^2;q^2)_k}q^{2k},
\end{equation*}
where $\tilde{\beta}_q(k)$ is the $a=q^n$ or $a=q^{-n}$ case of
$\beta_q(k)$ and satisfies $\tilde{\beta}_q(k)=0$ \:for\:
$(n+1)/2\leq k\leq n-1$. In the light of  the $d=m=2$ case of Lemma
\ref{lemma-b} again, we are led to
\begin{align*}
\sum_{i+j\leq n-1}\tilde{\beta}_q(i)\tilde{\beta}_q(j)
=[n]^2(bq)^{1-n}\frac{(bq^2;q^2)_{(n-1)/2}^2}{(q^2/b;q^2)_{(n-1)/2}^2}
\bigg\{\sum_{k=0}^{(n-1)/2}\frac{(b,bq,q^{1+n},q^{1-n};q^2)_k}{(q,q^2,bq^2,bq^2;q^2)_k}q^{2k}\bigg\}^2.
\end{align*}
Hence we find the $q$-supercongruence: modulo $(1-aq^n)(a-q^n)$,
\begin{align}
\sum_{i+j\leq n-1}\beta_q(i)\beta_q(j)
&\equiv[n]^2(bq)^{1-n}\frac{(bq^2;q^2)_{(n-1)/2}^2}{(q^2/b;q^2)_{(n-1)/2}^2}
\notag\\[1mm]
&\quad\times\bigg\{\sum_{k=0}^{(n-1)/2}\frac{(aq,q/a,b,bq;q^2)_k}{(q,q^2,bq^2,bq^2;q^2)_k}q^{2k}\bigg\}^2.
\label{eq:wei-m}
\end{align}
Interchanging the parameters $a$ and $b$ in \eqref{eq:wei-m}, there
holds the following $q$-supercongruence: modulo $(1-bq^n)(b-q^n)$,
\begin{align}
\sum_{i+j\leq n-1}\beta_q(i)\beta_q(j)
&\equiv[n]^2(aq)^{1-n}\frac{(aq^2;q^2)_{(n-1)/2}^2}{(q^2/a;q^2)_{(n-1)/2}^2}
\notag\\[1mm]
&\quad\times\bigg\{\sum_{k=0}^{(n-1)/2}\frac{(bq,q/b,a,aq;q^2)_k}{(q,q^2,aq^2,aq^2;q^2)_k}q^{2k}\bigg\}^2.
\label{eq:wei-n}
\end{align}

It is routine to see that the polynomials $(1-aq^n)(a-q^n)$,
$(1-bq^n)(b-q^n)$, and $[n]$ are relatively prime to one another.
 Paying attention to the relations
\begin{align*}
&\frac{(1-bq^n)(b-q^n)(-1-a^2+aq^n)}{(a-b)(1-ab)}\equiv1\pmod{(1-aq^n)(a-q^n)},
\\[1mm]
&\frac{(1-aq^n)(a-q^n)(-1-b^2+bq^n)}{(b-a)(1-ba)}\equiv1\pmod{(1-bq^n)(b-q^n)},
\end{align*}
and engaging the Chinese remainder theorem for coprime polynomials,
from \eqref{eq:wei-j}, \eqref{eq:wei-m}, and \eqref{eq:wei-n} we
deduce the following $q$-supercongruence: modulo
$[n](1-aq^n)(a-q^n)(1-bq^n)(b-q^n)$,
\begin{align}
\sum_{i+j\leq n-1}\beta_q(i)\beta_q(j) &\equiv
[n]^2(bq)^{1-n}\frac{(1-bq^n)(b-q^n)(-1-a^2+aq^n)}{(a-b)(1-ab)}
\notag\\[1mm]
&\quad\times\frac{(bq^2;q^2)_{(n-1)/2}^2}{(q^2/b;q^2)_{(n-1)/2}^2}
\bigg\{\sum_{k=0}^{(n-1)/2}\frac{(aq,q/a,b,bq;q^2)_k}{(q,q^2,bq^2,bq^2;q^2)_k}q^{2k}\bigg\}^2
\notag\\[1mm]
&\:+ [n]^2(aq)^{1-n}\frac{(1-aq^n)(a-q^n)(-1-b^2+bq^n)}{(b-a)(1-ba)}
\notag\\[1mm]
&\quad\times\frac{(aq^2;q^2)_{(n-1)/2}^2}{(q^2/a;q^2)_{(n-1)/2}^2}
\bigg\{\sum_{k=0}^{(n-1)/2}\frac{(bq,q/b,a,aq;q^2)_k}{(q,q^2,aq^2,aq^2;q^2)_k}q^{2k}\bigg\}^2.
\label{eq:wei-o}
\end{align}

The $b=1$ case of \eqref{eq:wei-o} can be written as the result:
modulo $[n]\Phi_n(q)^2(1-aq^n)(a-q^n)$,
\begin{align}
\sum_{i+j\leq n-1}\lambda_q(i)\lambda_q(j) &\equiv
[n]^2q^{1-n}\frac{(1-q^n)^2(1+a^2-aq^n)}{(1-a)^2}
\notag\\[1mm]
&\:- [n]^2(aq)^{1-n}\frac{(1-aq^n)(a-q^n)(2-q^n)}{(1-a)^2}
\notag\\[1mm]
&\quad\times\frac{(aq^2;q^2)_{(n-1)/2}^2}{(q^2/a;q^2)_{(n-1)/2}^2}
\bigg\{\sum_{k=0}^{(n-1)/2}\frac{(q,a,aq;q^2)_k}{(q^2,aq^2,aq^2;q^2)_k}q^{2k}\bigg\}^2.
\label{eq:wei-p}
\end{align}
Recall $q$-Saalsch\"{u}tz identity (cf. \cite[Appendix
(II.12)]{Gasper}):
\begin{align}
& _{3}\phi_{2}\!\left[\begin{array}{cccccccc}
 q^{-m}, &a, &b\\
  &c,  &q^{1-m}ab/c
\end{array};q,\, q \right]
=\frac{(c/a, c/b;q)_{m}}
 {(c, c/ab;q)_{m}}. \label{saal}
\end{align}
Employing the replacements $q\to q^2,b\to aq, c\to aq^2, m\to
(n-1)/2$ in \eqref{saal}, we catch hold of
\begin{align*}
\sum_{k=0}^{(n-1)/2}\frac{(a,aq,q^{1-n};q^2)_k}{(q^2,aq^2,aq^{2-n};q^2)_k}q^{2k}
=\frac{(q,q^2;q^2)_{(n-1)/2}}
 {(aq^2, q/a;q^2)_{(n-1)/2}}.
\end{align*}
The last equation may engenders
\begin{align}
\sum_{k=0}^{(n-1)/2}\frac{(q,a,aq;q^2)_k}{(q^2,aq^2,aq^{2};q^2)_k}q^{2k}
\equiv\frac{(q,q^2;q^2)_{(n-1)/2}}
 {(aq^2, q/a;q^2)_{(n-1)/2}}\pmod{\Phi_n(q)}.
\label{eq:wei-q}
\end{align}
Substitute \eqref{eq:wei-q} into \eqref{eq:wei-p} to demonstrate the
following conclusion: modulo $[n]\Phi_n(q)^2(1-aq^n)(a-q^n)$,
\begin{align}
\sum_{i+j\leq n-1}\lambda_q(i)\lambda_q(j) &\equiv
[n]^2q^{1-n}\frac{(1-q^n)^2(1+a^2-aq^n)}{(1-a)^2}
\notag\\[1mm]
&- [n]^2(aq)^{1-n}\frac{(1-aq^n)(a-q^n)(2-q^n)}{(1-a)^2}
\frac{(q,q^2;q^2)_{(n-1)/2}^2}{(q/a,q^2/a;q^2)_{(n-1)/2}^2}
\notag\\[1mm]
&\equiv [n]^2q^{1-n}\frac{(1-q^n)^2(1+a^2-aq^n)}{(1-a)^2}
\notag\\[1mm]
&- [n]^2q^{1-n}\frac{(1-aq^n)(a-q^n)(2-q^n)}{(1-a)^2}
\frac{(q^2;q^2)_{(n-1)/2}^4}{(aq^2,q^2/a;q^2)_{(n-1)/2}^2}.
\label{eq:wei-r}
\end{align}
In view of the identity
\begin{align}
(1-q^n)^2(1+a^2-aq^n)=(1-a)^2+(1-aq^n)(a-q^n)(2-q^n),
\label{relation}
\end{align}
\eqref{eq:wei-r} can be rewritten as follows: modulo
$[n]\Phi_n(q)^2(1-aq^n)(a-q^n)$,
\begin{align*}
\sum_{i+j\leq n-1}\lambda_q(i)\lambda_q(j) &\equiv
[n]^2q^{1-n}+[n]^2q^{1-n}\frac{(1-aq^n)(a-q^n)(2-q^n)}{(1-a)^2}
\notag\\[1mm]
&\quad\times\bigg\{1-\frac{(q^2;q^2)_{(n-1)/2}^4}{(aq^2,q^2/a;q^2)_{(n-1)/2}^2}\bigg\}
\\[1mm]
&\equiv [n]^2q^{1-n}+[n]^2q^{1-n}\frac{(1-aq^n)(a-q^n)}{(1-a)^2}
\notag\\[1mm]
&\quad\times\bigg\{1-\frac{(q^2;q^2)_{(n-1)/2}^4}{(aq^2,q^2/a;q^2)_{(n-1)/2}^2}\bigg\},
\end{align*}
where we have utilized the relation $q^n\equiv1\pmod{\Phi_n(q)}$ in
the last step. This completes the proof of Theorem \ref{thm-d}.
\end{proof}

Subsequently, we begin to prove Theorem \ref{thm-a}.

\begin{proof}[Proof of Theorem \ref{thm-a}]
By the aid of L'H\^{o}pital rule, we derive
\begin{align*}
\lim_{a\to1}\frac{1}{(1-a)^2}\bigg\{1-\frac{(q^2;q^2)_{(n-1)/2}^4}{(aq^2,q^2/a;q^2)_{(n-1)/2}^2}\bigg\}
=-2\sum_{t=1}^{(n-1)/2}\frac{q^{2t}}{(1-q^{2t})^2}.
\end{align*}
Letting $a\to1$ in \eqref{eq:wei-d} and employing this limit, we
prove \eqref{eq:wei-a}.
\end{proof}

\section{Proof of Theorem \ref{thm-b}}

For the aim to prove Theorem \ref{thm-b}, we demand Lemma
\ref{lemma-b} and the following two Lemmas.

\begin{lemma}\label{lemma-c}
Let $n$ be a positive odd integer. Then
\begin{align*}
\sum_{k=0}^{(n-1)/2}[4k+1]\frac{(aq,q/a,bq,q/b,q/c,q;q^2)_k}{(q^2/a,aq^2,q^2/b,bq^2,cq^2,q^2;q^2)_k}(cq)^k
\equiv 0\pmod{[n]}.
\end{align*}
\end{lemma}

\begin{proof}
Ni and Wang \cite[Lemma 2.2]{NW}) shows that
\begin{align}
&\sum_{k=0}^{\mu}[2dk+r]\frac{(aq^r,q^r/a,bq^r,q^r/b,q^r/c,q^r;q^d)_k}{(q^d/a,aq^d,q^d/b,bq^d,cq^d,q^d;q^d)_k}(cq^{2d-3r})^k
\equiv 0\pmod{[n]}, \label{eq:NW-a}
\end{align}
where $d,n$ are positive integers and $r$ is an integer with $0\leq
\mu\leq n-1$, $\gcd(n,d)=1$, and $d\mu\equiv -r\pmod{n}$. Choosing
$d=2$, $\mu=(n-1)/2$, and $r=1$ in \eqref{eq:NW-a}, we obtain the
desired formula.
\end{proof}

\begin{lemma}\label{lemma-d}
For three polynomials $(1-aq^n)(a-q^n)$, $(1-bq^n)(b-q^n)$, and
$(c-q^n)$, there are the following relations:
\begin{align*}
&\frac{(1-bq^n)(b-q^n)(c-q^n)\{-1-a^2-a^4+ac+a^3c+a(1+a^2-ac)q^n\}}{(1-ba)(b-a)(1-ac)(a-c)}
\\[1mm]
&\:\:\equiv1\pmod{(1-aq^n)(a-q^n)},
\\[1mm]
&\frac{(1-aq^n)(a-q^n)(c-q^n)\{-1-b^2-b^4+bc+b^3c+b(1+b^2-bc)q^n\}}{(1-ab)(a-b)(1-bc)(b-c)}
\\[1mm]
&\:\:\equiv1\pmod{(1-bq^n)(b-q^n)},
\\[1mm]
&\qquad\quad\frac{(1-aq^n)(a-q^n)(1-bq^n)(b-q^n)}{(1-ac)(a-c)(1-bc)(b-c)}\equiv1\pmod{(c-q^n)}.
\end{align*}
\end{lemma}

At the moment, we shall present the following parametric
generalization of Theorem \ref{thm-b}.

\begin{theorem}\label{thm-e}
Let $n$ be a positive odd integer and
$$\lambda_q(k)=[4k+1]\frac{(q;q^2)_k^4(aq,q/a;q^2)_k}{(q^2;q^2)_k^4(q^2/a,aq^2;q^2)_k}q^{k}\quad\text{with}\quad k\geq
  0.$$
 Then, modulo
$[n]\Phi_n(q)^3(1-aq^n)(a-q^n)$,
\begin{align}
&\sum_{i+j\leq n-1}\lambda_q(i)\lambda_q(j)
\notag\\
&\quad\equiv[n]^2q^{1-n}\frac{(1-q^n)^2(1+a^2-aq^n)}{(1-a)^2}\bigg\{\sum_{k=0}^{(n-1)/2}\frac{(aq,q/a;q^2)_k(q;q^2)_k^2}{(q^2;q^2)_k^4}q^{2k}\bigg\}^2
\notag\\[1mm]
&\quad\,-[n]^2q^{1-n}\frac{(2-q^n)(1-aq^n)(a-q^n)}{(1-a)^2}\bigg\{\sum_{k=0}^{(n-1)/2}\frac{(q;q^2)_k^4}{(aq^2,q^2/a;q^2)_k(q^2;q^2)_k^2}q^{2k}\bigg\}^2.
\label{eq:wei-aa}
\end{align}
\end{theorem}

\begin{proof}
It is obvious that the $n=1$ case of \eqref{eq:wei-aa} is right. We
now posit that $n>1$.  Set
$$\beta_q(k)=[4k+1]\frac{(aq,q/a,bq,q/b,q/c,q;q^2)_k}{(q^2/a,aq^2,q^2/b,bq^2,cq^2,q^2;q^2)_k}(cq)^k.$$
Let $\zeta\neq 1$ be an $n$-th root of unity. That is to say,
$\zeta$ is a primitive $u$-th root of unity with $u\mid n$.
Considering that $(q;q^2)_k$ appears in the numerator of
$\beta_q(k)$, so $\beta_\zeta(k)=0$ for $(u+1)/2\leqslant k\leqslant
u-1$. In terms of the $d=m=2$ case of Lemma \ref{lemma-b} and the
$n=u$ case of Lemma \ref{lemma-c}, we have
$$
\sum_{i+j\leq
u-1}\beta_\zeta(i)\beta_\zeta(j)=\bigg\{\sum_{k=0}^{(u-1)/2}\beta_\zeta(k)\bigg\}^2=0.
$$
For any integers $k$ and $l$ with $0\leqslant k\leqslant u-1$, it is
not difficult to verify that
$$\frac{\beta_\zeta(lu+k)}{\beta_\zeta(lu)}=\beta_\zeta(k).$$
By means of the $d=m=2$ case of Lemma \ref{lemma-b} again, we get
\begin{align*}
\sum_{i+j\leq n-1}\beta_\zeta(i)\beta_\zeta(j)
&=\sum_{t=0}^{n-1}\sum_{j=0}^{t}\beta_\zeta(j)\beta_\zeta(t-j)\\
&=\sum_{l=0}^{n/u-1}\sum_{k=0}^{u-1}\sum_{j=0}^{lu+k}\beta_\zeta(j)\beta_\zeta(lu+k-j)
\\[1mm]
&=\sum_{l=0}^{n/u-1}\sum_{k=0}^{u-1}\sum_{i=0}^{l}\beta_\zeta(iu)\beta_\zeta((l-i)u)\sum_{j=0}^{k}\beta_\zeta(j)\beta_\zeta(k-j)
\\[1mm]
&=\sum_{l=0}^{n/u-1}\sum_{i=0}^{l}\beta_\zeta(iu)\beta_\zeta((l-i)u)\sum_{k=0}^{u-1}\sum_{j=0}^{k}\beta_\zeta(j)\beta_\zeta(k-j)
\\[1mm]
&=0.
\end{align*}
Since the upper equality is true for any $n$-th root of unit
$\zeta\neq 1$, we judge that
\begin{equation}\label{eq:wei-bb}
\sum_{i+j\leq n-1}\beta_q(i)\beta_q(j)\equiv 0\pmod{[n]}.
\end{equation}

Performing the replacements $q\to q^2,a\to q, b\to bq, c\to q/c,
d\to q/b, e\to q^{1+n},m\to (n-1)/2$ in \eqref{eq:watson}, there is
\begin{equation*}
\sum_{k=0}^{(n-1)/2}\tilde{\beta}_q(k)=[n](b/q)^{(n-1)/2}\frac{(q^2/b;q^2)_{(n-1)/2}}{(bq^2;q^2)_{(n-1)/2}}
\sum_{k=0}^{(n-1)/2}\frac{(cq/b,q/b,q^{1+n},q^{1-n};q^2)_k}{(q^2,q^2/b,q^2/b,cq^2;q^2)_k}q^{2k},
\end{equation*}
where $\tilde{\beta}_q(k)$ means the $a=q^n$ or $a=q^{-n}$ case of
$\beta_q(k)$ subject to $\tilde{\beta}_q(k)=0$ \:for\: $(n+1)/2\leq
k\leq n-1$. Via the $d=m=2$ case of Lemma \ref{lemma-b}, it is
ordinary to understand that
\begin{align*}
\sum_{i+j\leq
n-1}\tilde{\beta}_q(i)\tilde{\beta}_q(j)&=[n]^2(b/q)^{n-1}\frac{(q^2/b;q^2)_{(n-1)/2}^2}{(bq^2;q^2)_{(n-1)/2}^2}
\notag\\
&\quad\times
\bigg\{\sum_{k=0}^{(n-1)/2}\frac{(cq/b,q/b,q^{1+n},q^{1-n};q^2)_k}{(q^2,q^2/b,q^2/b,cq^2;q^2)_k}q^{2k}\bigg\}^2.
\end{align*}
 Then there holds the $q$-supercongruence: modulo $(1-aq^n)(a-q^n)$,
\begin{align}
\sum_{i+j\leq
n-1}{\beta}_q(i){\beta}_q(j)&\equiv[n]^2(b/q)^{n-1}\frac{(q^2/b;q^2)_{(n-1)/2}^2}{(bq^2;q^2)_{(n-1)/2}^2}
\notag\\
&\quad\times\bigg\{\sum_{k=0}^{(n-1)/2}\frac{(aq,q/a,q/b,cq/b;q^2)_k}{(q^2,q^2/b,q^2/b,cq^2;q^2)_k}q^{2k}\bigg\}^2.
\label{eq:wei-cc}
\end{align}
Interchanging the parameters $a$ and $b$ in \eqref{eq:wei-cc}, we
discover the dual form: modulo $(1-bq^n)(b-q^n)$,
\begin{align}
\sum_{i+j\leq
n-1}{\beta}_q(i){\beta}_q(j)&\equiv[n]^2(a/q)^{n-1}\frac{(q^2/a;q^2)_{(n-1)/2}^2}{(aq^2;q^2)_{(n-1)/2}^2}
\notag\\
&\quad\times\bigg\{\sum_{k=0}^{(n-1)/2}\frac{(bq,q/b,q/a,cq/a;q^2)_k}{(q^2,q^2/a,q^2/a,cq^2;q^2)_k}q^{2k}\bigg\}^2.
\label{eq:wei-dd}
\end{align}

Employing the replacements $q\to q^2,a\to q, b\to bq, c\to q/b, d\to
aq, e\to q/a,m\to (n-1)/2$ in \eqref{eq:watson}, it is easy to
realize that
\begin{equation*}
\sum_{k=0}^{(n-1)/2}\hat{\beta}_q(k)=[n]\frac{(q;q^2)_{(n-1)/2}^2}{(aq^2,q^2/a;q^2)_{(n-1)/2}}
\sum_{k=0}^{(n-1)/2}\frac{(q,aq,q/a,q^{1-n};q^2)_k}{(q^2,bq^2,q^2/b,q^{2-n};q^2)_k}q^{2k},
\end{equation*}
where $\hat{\beta}_q(k)$ indicates the $c=q^n$ case of $\beta_q(k)$
and meets $\hat{\beta}_q(k)=0$ \:for\: $(n+1)/2\leq k\leq n-1$.
Therefore, the $d=m=2$ case of Lemma \ref{lemma-b} gives
\begin{align*}
\sum_{i+j\leq
n-1}\hat{\beta}_q(i)\hat{\beta}_q(j)&=[n]^2\frac{(q;q^2)_{(n-1)/2}^4}{(aq^2,q^2/a;q^2)_{(n-1)/2}^2}
\notag\\
&\quad\times
\bigg\{\sum_{k=0}^{(n-1)/2}\frac{(q,aq,q/a,q^{1-n};q^2)_k}{(q^2,bq^2,q^2/b,q^{2-n};q^2)_k}q^{2k}\bigg\}^2.
\end{align*}
 Thus the last equation produces the $q$-congruence: modulo $(c-q^n)$,
\begin{align}
\sum_{i+j\leq
n-1}{\beta}_q(i){\beta}_q(j)&\equiv[n]^2\frac{(q;q^2)_{(n-1)/2}^4}{(aq^2,q^2/a;q^2)_{(n-1)/2}^2}
\notag\\
&\quad\times\bigg\{\sum_{k=0}^{(n-1)/2}\frac{(q,aq,q/a,q/c;q^2)_k}{(q^2,bq^2,q^2/b,q^{2}/c;q^2)_k}q^{2k}\bigg\}^2.
\label{eq:wei-ee}
\end{align}

It is routine to check that the polynomials $(1-aq^n)(a-q^n)$,
$(1-bq^n)(b-q^n)$, $(c-q^n)$, and $[n]$ are relatively prime to one
another. Engaging the Chinese remainder theorem for coprime
polynomials, from Lemma \ref{lemma-d} and
\eqref{eq:wei-bb}-\eqref{eq:wei-ee}, we deduce the following
$q$-supercongruence: modulo
$[n](1-aq^n)(a-q^n)(1-bq^n)(b-q^n)(c-q^n)$,
\begin{align}
&\sum_{i+j\leq n-1}{\beta}_q(i){\beta}_q(j)
\notag\\
&\quad\equiv\frac{(1-bq^n)(b-q^n)(c-q^n)\{-1-a^2-a^4+ac+a^3c+a(1+a^2-ac)q^n\}}{(1-ba)(b-a)(1-ac)(a-c)}
\notag\\[1mm]
&\qquad\times[n]^2(b/q)^{n-1}\frac{(q^2/b;q^2)_{(n-1)/2}^2}{(bq^2;q^2)_{(n-1)/2}^2}
\bigg\{\sum_{k=0}^{(n-1)/2}\frac{(aq,q/a,q/b,cq/b;q^2)_k}{(q^2,q^2/b,q^2/b,cq^2;q^2)_k}q^{2k}\bigg\}^2
\notag
\end{align}
\begin{align}
&\quad+\frac{(1-aq^n)(a-q^n)(c-q^n)\{-1-b^2-b^4+bc+b^3c+b(1+b^2-bc)q^n\}}{(1-ab)(a-b)(1-bc)(b-c)}
\notag\\[1mm]
&\qquad\times[n]^2(a/q)^{n-1}\frac{(q^2/a;q^2)_{(n-1)/2}^2}{(aq^2;q^2)_{(n-1)/2}^2}
\bigg\{\sum_{k=0}^{(n-1)/2}\frac{(bq,q/b,q/a,cq/a;q^2)_k}{(q^2,q^2/a,q^2/a,cq^2;q^2)_k}q^{2k}\bigg\}^2
\notag\\[1mm]
&\quad+\frac{(1-aq^n)(a-q^n)(1-bq^n)(b-q^n)}{(1-ac)(a-c)(1-bc)(b-c)}
\notag\\[1mm]
&\qquad\times[n]^2\frac{(q;q^2)_{(n-1)/2}^4}{(aq^2,q^2/a;q^2)_{(n-1)/2}^2}
\bigg\{\sum_{k=0}^{(n-1)/2}\frac{(q,aq,q/a,q/c;q^2)_k}{(q^2,bq^2,q^2/b,q^{2}/c;q^2)_k}q^{2k}\bigg\}^2.
\label{eq:wei-ff}
\end{align}
Fixing $c=1$ in \eqref{eq:wei-ff}, we gain the formula: modulo
$[n]\Phi_n(q)(1-aq^n)(a-q^n)(1-bq^n)(b-q^n)$,
\begin{align*}
&\sum_{i+j\leq n-1}\bar{\beta}_q(i)\bar{\beta}_q(j)
\notag\\
&\quad\equiv\frac{(1-bq^n)(b-q^n)(1+a^2-aq^n)}{(1-ba)(b-a)}
\notag\\[1mm]
&\qquad\times[n]^2(b/q)^{n-1}\frac{(q^2/b;q^2)_{(n-1)/2}^2}{(bq^2;q^2)_{(n-1)/2}^2}
\bigg\{\sum_{k=0}^{(n-1)/2}\frac{(aq,q/a,q/b,q/b;q^2)_k}{(q^2,q^2,q^2/b,q^2/b;q^2)_k}q^{2k}\bigg\}^2
\notag\\[1mm]
 &\quad+\frac{(1-aq^n)(a-q^n)(1+b^2-bq^n)}{(1-ab)(a-b)}
\notag\\[1mm]
&\qquad\times[n]^2(a/q)^{n-1}\frac{(q^2/a;q^2)_{(n-1)/2}^2}{(aq^2;q^2)_{(n-1)/2}^2}
\bigg\{\sum_{k=0}^{(n-1)/2}\frac{(bq,q/b,q/a,q/a;q^2)_k}{(q^2,q^2,q^2/a,q^2/a;q^2)_k}q^{2k}\bigg\}^2,
\end{align*}
where $\bar{\beta}_q(k)$ signifies the $c=1$ case of $\beta_q(k)$.
The $b=1$ case of it can be stated as follows: modulo
$[n]\Phi_n(q)^3(1-aq^n)(a-q^n)$,
\begin{align}
&\sum_{i+j\leq n-1}\lambda_q(i)\lambda_q(j)
\notag\\
&\quad\equiv [n]^2q^{1-n}\frac{(1-q^n)^2(1+a^2-aq^n)}{(1-a)^2}
\bigg\{\sum_{k=0}^{(n-1)/2}\frac{(aq,q/a;q^2)_k(q;q^2)_k^2}{(q^2;q^2)_k^4}q^{2k}\bigg\}^2
\notag\\[1mm]
&\quad-[n]^2(a/q)^{n-1}\frac{(2-q^n)(1-aq^n)(a-q^n)}{(1-a)^2}
\notag\\[1mm]
&\qquad\times\frac{(q^2/a;q^2)_{(n-1)/2}^2}{(aq^2;q^2)_{(n-1)/2}^2}
\bigg\{\sum_{k=0}^{(n-1)/2}\frac{(q/a,q;q^2)_k^2}{(q^2/a,q^2;q^2)_k^2}q^{2k}\bigg\}^2.
\label{eq:wei-gg}
\end{align}

Recollect Sear's $_4\phi_3$ transformation (cf. \cite[Appendix
(III.15)]{Gasper})
\begin{align}
 _{4}\phi_{3}\!\left[\begin{array}{cccccccc}
a,& b,& c, & q^{-m} \\
  & d, & e,  & f
\end{array};q,\, q
\right]&=a^m\frac{(e/a, f/a;q)_{m}} {(e, f;q)_{m}}
\notag\\[1mm]
&\quad\times{_{4}\phi_{3}}\!\left[\begin{array}{cccccccc}
a,& d/b,& d/c, & q^{-m} \\
  & d, & aq^{1-m}/e,  & aq^{1-m}/f
\end{array};q,\, q
\right], \label{eq:Sear}
\end{align}
where $def=abcq^{1-m}$. Performing the replacements $q\to q^2,a\to
q^{1+n}, b\to q/a, c\to q/a, d\to q^2/a, e\to q^2,  f\to q^2/a, m\to
(n-1)/2$ in \eqref{eq:Sear}, we arrive at

\begin{align*}
 _{4}\phi_{3}\!\left[\begin{array}{cccccccc}
q/a,& q/a,& q^{1+n}, & q^{1-n} \\
  & q^2, & q^2/a,  & q^2/a
\end{array};q^2,\, q^2
\right]&=a^{(1-n)/2}\frac{(aq^2;q^2)_{(n-1)/2}}{(q^2/a;q^2)_{(n-1)/2}}
\notag\\
&\quad\times{_{4}\phi_{3}}\!\left[\begin{array}{cccccccc}
q,& q,& q^{1+n}, & q^{1-n} \\
  & q^2, & aq^2,  & q^2/a
\end{array};q^2,\, q^2
\right].
\end{align*}
Noting that
$(q^{1+n},q^{1-n};q^2)_k\equiv(q;q^2)_k^2\pmod{\Phi_n(q)^2}$, we
have
\begin{align}
\sum_{k=0}^{(n-1)/2}\frac{(q/a,q;q^2)_k^2}{(q^2/a,q^2;q^2)_k^2}q^{2k}
&\equiv
a^{(1-n)/2}\frac{(aq^2;q^2)_{(n-1)/2}}{(q^2/a;q^2)_{(n-1)/2}}
\notag\\[1mm]
&\quad\times\sum_{k=0}^{(n-1)/2}\frac{(q;q^2)_k^4}{(aq^2,q^2/a;q^2)_k(q^2;q^2)_k^2}q^{2k}
\pmod{\Phi_n(q)^2}. \label{eq:wei-hh}
\end{align}
Substituting \eqref{eq:wei-hh} into \eqref{eq:wei-gg}, we are led to
\eqref{eq:wei-aa} to finish the proof.
\end{proof}

Afterwards, we start to prove Theorem \ref{thm-b}.

\begin{proof}[Proof of Theorem \ref{thm-b}]
Through the relation \eqref{relation}, \eqref{eq:wei-aa} can be
manipulated as the following form: modulo
$[n]\Phi_n(q)^3(1-aq^n)(a-q^n)$,
\begin{align}
\sum_{i+j\leq n-1}\lambda_q(i)\lambda_q(j)
&\equiv[n]^2q^{1-n}\bigg\{\sum_{k=0}^{(n-1)/2}\frac{(aq,q/a;q^2)_k(q;q^2)_k^2}{(q^2;q^2)_k^4}q^{2k}\bigg\}^2
\notag\\[1mm]
&\quad\,+[n]^2q^{1-n}\frac{(2-q^n)(1-aq^n)(a-q^n)}{(1-a)^2}\Omega_n(a),
\label{eq:wei-ii}
\end{align}
where
\begin{align*}
\Omega_n(a)&=
\bigg\{\sum_{k=0}^{(n-1)/2}\frac{(aq,q/a;q^2)_k(q;q^2)_k^2}{(q^2;q^2)_k^4}q^{2k}\bigg\}^2
\notag\\[1mm]
&\quad-\bigg\{\sum_{k=0}^{(n-1)/2}\frac{(q;q^2)_k^4}{(aq^2,q^2/a;q^2)_k(q^2;q^2)_k^2}q^{2k}\bigg\}^2.
\end{align*}
Thanks to L'H\^{o}pital rule, we derive
\begin{align*}
\lim_{a\to1}\frac{\Omega_n(a)}{(1-a)^2}
=-2\sum_{k=0}^{(n-1)/2}\frac{(q;q^2)_k^4}{(q^2;q^2)_k^4}q^{2k}
\sum_{k=0}^{(n-1)/2}\frac{(q;q^2)_k^4}{(q^2;q^2)_k^4}q^{2k}
\sum_{t=1}^{2k}\frac{q^t}{(1-q^t)^2}.
\end{align*}
Letting $a\to1$ in \eqref{eq:wei-ii} and using the above limit, we
find \eqref{eq:wei-b}.
\end{proof}

\section{Proof of Theorem \ref{thm-c}}

In order to prove Theorem \ref{thm-c}, we need Lemma \ref{lemma-d}
and the following two lemmas.

\begin{lemma}\label{lemma-e} Let $d,m,n,r$ be positive integers
subject to $n\equiv r\pmod{d}$, $(r,d)=1$, and $d\geq m\geq2$. Let
$\{\lambda(k)\}_{k=0}^\infty$ be a complex sequence. If
$\lambda(k)=0$ for $(n+d-r)/d\leq k\leq n-1$, then
$$
\sum_{k_1+k_2+\cdots+k_m\leq
n-1}\lambda(k_1)\lambda(k_2)\cdots\lambda(k_m)
=\Bigg\{\sum_{k=0}^{(n-r)/d}\lambda(k)\Bigg\}^m.
$$
\end{lemma}

\begin{proof}
Similar to the proof of Lemma \ref{lemma-b}, it is not difficult to
show that
$$
\sum_{k_1+\cdots+k_m\leq
n-r}\lambda(k_1)\lambda(k_2)\cdots\lambda(k_m)
=\Bigg\{\sum_{k=0}^{(n-r)/d}\lambda(k)\Bigg\}^m.
$$
When $n-r+1\leq k_1+k_2+\cdots+k_m\leq n-1,$ there exists at least a
$k_i$ such that $k_i\geq(n+d-r)/d.$ Thus we have
$$
 \sum_{n-r+1\leq
k_1+k_2+\cdots+k_m\leq n-1}
\lambda(k_1)\lambda(k_2)\cdots\lambda(k_m) =0.
$$
The sum of the last two equalities produces the desired result.
\end{proof}

\begin{lemma}\label{lemma-f}
Let $n$ be a positive odd integer. Then
\begin{align*}
\sum_{k=0}^{(n-r)/d}[2dk+r]\frac{(aq^r,q^r/a,bq^r,q^r/b,q^r/c,q^r;q^d)_k}{(q^d/a,aq^d,q^d/b,bq^d,cq^d,q^d;q^d)_k}(cq^{2d-3r})^k
\equiv 0\pmod{[n]}.
\end{align*}
\end{lemma}

\begin{proof}
 Fixing $\mu=(n-r)/d$ in \eqref{eq:NW-a}, we obtain the expected conclusion.
\end{proof}

Currently we are about to furnish a  parametric generalization of
Theorem \ref{thm-c}.

\begin{theorem}\label{thm-f}
Let $d,m,n,r$ be positive integers subject to $n\equiv r\pmod{d}$,
$(r,d)=1$, and $d\geq m\geq2$. Let $\lambda_q(k)$ represent
$$\lambda_q(k)=[2dk+r]\frac{(q^r;q^d)_k^4(aq^r,q^r/a;q^d)_k}{(q^d;q^d)_k^4(q^d/a,aq^d;q^d)_k}q^{(2d-3r)k}\quad\text{with}\quad k\geq
  0.$$
 Then, modulo $\Phi_n(q)^4(1-aq^n)(a-q^n)$,
\begin{align}
&\sum_{i_1+i_2+\cdots+i_m\leq n-1}\lambda_q(i_1)\cdots\lambda_q(i_m)
\notag\\[1mm]
&\quad\equiv
[n]^mq^{mr(r-n)/d}\frac{(1-q^n)^2(1+a^2-aq^n)}{(1-a)^2}\frac{(q^{2r};q^d)_{(n-r)/d}^m}{(q^{d};q^d)_{(n-r)/d}^m}
\notag\\[1mm]
&\qquad\times
\bigg\{\sum_{k=0}^{(n-r)/d}\frac{(aq^r,q^r/a,q^r,q^{d-r};q^d)_k}{(q^d;q^d)_k^3(q^{2r};q^d)_k}q^{dk}\bigg\}^{m}
\notag
\\[1mm]
&\quad-[n]^mq^{mr(r-n)/d}\frac{(2-q^n)(1-aq^n)(a-q^n)}{(1-a)^2}\frac{(q^{2r};q^d)_{(n-r)/d}^m}{(q^{d};q^d)_{(n-r)/d}^m}
\notag\\[1mm]
&\qquad\times
\bigg\{\sum_{k=0}^{(n-r)/d}\frac{(q^r;q^d)_k^3(q^{d-r};q^d)_k}{(aq^d,q^d/a,q^d,q^{2r};q^d)_k}q^{dk}\bigg\}^{m}.
\label{eq:wei-aaa}
\end{align}
\end{theorem}

\begin{proof}
 Set
$$\beta_q(k)=[2dk+r]\frac{(aq^r,q^r/a,bq^r,q^r/b,q^r/c,q^r;q^d)_k}{(q^d/a,aq^d,q^d/b,bq^d,cq^d,q^d;q^d)_k^4}(cq)^{(2d-3r)k}.$$
Firstly, we hope to prove the following $q$-congruence:
\begin{equation}\label{eq:wei-bbb}
\sum_{i_1+i_2+\cdots+i_m\leq n-1}\beta_q(i_1)\cdots\beta_q(i_m)
\equiv 0\pmod{\Phi_n(q)}.
\end{equation}
It is obvious that the $n=1$ case of \eqref{eq:wei-bb} is right. We
now posit that $n>1$. Let $\xi\neq 1$ be a primitive $n$-th root of
unity. Considering that $(q^r;q^d)_k$ arises in the numerator of
$\beta_q(k)$, so $\beta_\zeta(k)=0$ for $(n+d-r)/d\leq k\leq n-1$.
Via Lemmas \ref{lemma-e} and \ref{lemma-f}, we have
$$
\sum_{i_1+i_2+\cdots+i_m\leq n-1
}\beta_\zeta(i_1)\cdots\beta_\zeta(i_m)=\Bigg\{\sum_{k=0}^{(n-r)/d}\beta_\zeta(k)\Bigg\}^m=0.
$$
Noticing that the upper equation is correct for any $n$-th root of
unit $\zeta\neq 1$, we know the correctness of \eqref{eq:wei-bbb}.

Secondly, performing the replacements $q\to q^d,a\to q^r, b\to bq^r,
c\to q^r/c, d\to q^r/b, e\to q^{r+n},m\to (n-r)/d$ in
\eqref{eq:watson}, there is
\begin{equation*}
\sum_{k=0}^{(n-r)/d}\tilde{\beta}_q(k)=[n](b/q^r)^{(n-r)/d}\frac{(q^{2r}/b;q^d)_{(n-r)/d}}{(bq^d;q^d)_{(n-r)/d}}
\sum_{k=0}^{(n-r)/d}\frac{(cq^{d-r}/b,q^r/b,q^{r+n},q^{r-n};q^d)_k}{(q^d,q^d/b,q^{2r}/b,cq^d;q^d)_k}q^{dk},
\end{equation*}
where $\tilde{\beta}_q(k)$ means the $a=q^n$ or $a=q^{-n}$ case of
$\beta_q(k)$ subject to $\tilde{\beta}_q(k)=0$ \:for\:
$(n+d-r)/d\leq k\leq n-1$. Via Lemma \ref{lemma-e}, it is ordinary
to observe that
\begin{align*}
\sum_{i_1+i_2+\cdots+i_m\leq
n-1}\tilde{\beta}_q(i_1)\cdots\tilde{\beta}_q(i_m) &\equiv
[n]^m(b/q^r)^{m(n-r)/d}\frac{(q^{2r}/b;q^d)_{(n-r)/d}^m}{(bq^d;q^d)_{(n-r)/d}^m}
\notag\\[1mm]&\quad
\times
\bigg\{\sum_{k=0}^{(n-r)/d}\frac{(cq^{d-r}/b,q^r/b,q^{r+n},q^{r-n};q^d)_k}{(q^d,q^d/b,q^{2r}/b,cq^d;q^d)_k}q^{dk}\bigg\}^{m}.
\end{align*}
 Then there holds the $q$-supercongruence: modulo $(1-aq^n)(a-q^n)$,
\begin{align}
\sum_{i_1+i_2+\cdots+i_m\leq n-1}\beta_q(i_1)\cdots\beta_q(i_m)
&\equiv
[n]^m(b/q^r)^{m(n-r)/d}\frac{(q^{2r}/b;q^d)_{(n-r)/d}^m}{(bq^d;q^d)_{(n-r)/d}^m}
\notag\\[1mm]&\quad \times
\bigg\{\sum_{k=0}^{(n-r)/d}\frac{(aq^r,q^r/a,q^r/b,cq^{d-r}/b;q^d)_k}{(q^d,q^d/b,q^{2r}/b,cq^d;q^d)_k}q^{dk}\bigg\}^{m}.
\label{eq:wei-ccc}
\end{align}
Interchanging the parameters $a$ and $b$ in \eqref{eq:wei-ccc}, we
discover the dual form: modulo $(1-bq^n)(b-q^n)$,
\begin{align}
\sum_{i_1+i_2+\cdots+i_m\leq n-1}\beta_q(i_1)\cdots\beta_q(i_m)
&\equiv
[n]^m(a/q^r)^{m(n-r)/d}\frac{(q^{2r}/a;q^d)_{(n-r)/d}^m}{(aq^d;q^d)_{(n-r)/d}^m}
\notag\\[1mm]&\quad \times
\bigg\{\sum_{k=0}^{(n-r)/d}\frac{(bq^r,q^r/b,q^r/a,cq^{d-r}/a;q^d)_k}{(q^d,q^d/a,q^{2r}/a,cq^d;q^d)_k}q^{dk}\bigg\}^{m}.
\label{eq:wei-ddd}
\end{align}

Employing the replacements $q\to q^d,a\to q^r, b\to bq^r, c\to
q^r/b, d\to aq^r, e\to q^{r}/a,m\to (n-r)/d$ in \eqref{eq:watson},
it is easy to realize that
\begin{equation*}
\sum_{k=0}^{(n-r)/d}\hat{\beta}_q(k)=[n]\frac{(q^r,q^{d-r};q^d)_{(n-r)/d}}{(aq^d,q^d/a;q^d)_{(n-r)/d}}
\sum_{k=0}^{(n-r)/d}\frac{(q^{d-r},aq^r,q^{r}/a,q^{r-n};q^d)_k}{(q^d,q^d/b,bq^d,q^{2r-n};q^d)_k}q^{dk},
\end{equation*}
where $\hat{\beta}_q(k)$ indicates the $c=q^n$ case of $\beta_q(k)$
and meets $\hat{\beta}_q(k)=0$ \:for\: $(n+d-r)/d\leq k\leq n-1$.
Therefore, Lemma \ref{lemma-f} gives
\begin{align*}
\sum_{i_1+i_2+\cdots+i_m\leq
n-1}\hat{\beta}_q(i_1)\cdots\hat{\beta}_q(i_m) &\equiv
[n]^m\frac{(q^r,q^{d-r};q^d)_{(n-r)/d}^m}{(aq^d,q^d/a;q^d)_{(n-r)/d}^m}
\notag\\[1mm]&\quad
\times
\bigg\{\sum_{k=0}^{(n-r)/d}\frac{(q^{d-r},aq^r,q^{r}/a,q^{r-n};q^d)_k}{(q^d,q^d/b,bq^d,q^{2r-n};q^d)_k}q^{dk}\bigg\}^{m}.
\end{align*}
 Then there holds the $q$-supercongruence: modulo $(c-q^n)$,
\begin{align}
\sum_{i_1+i_2+\cdots+i_m\leq n-1}\beta_q(i_1)\cdots\beta_q(i_m)
&\equiv
[n]^m\frac{(q^r,q^{d-r};q^d)_{(n-r)/d}^m}{(aq^d,q^d/a;q^d)_{(n-r)/d}^m}
\notag\\[1mm]&\quad \times
\bigg\{\sum_{k=0}^{(n-r)/d}\frac{(q^{d-r},aq^r,q^{r}/a,q^{r}/c;q^d)_k}{(q^d,q^d/b,bq^d,q^{2r}/c;q^d)_k}q^{dk}\bigg\}^{m}.
\label{eq:wei-eee}
\end{align}

It is routine to check that the polynomials $(1-aq^n)(a-q^n)$,
$(1-bq^n)(b-q^n)$, $(c-q^n)$, and $\Phi_n(q)$ are relatively prime
to one another. Engaging the Chinese remainder theorem for coprime
polynomials, from Lemma \ref{lemma-d} and
\eqref{eq:wei-bbb}-\eqref{eq:wei-eee}, we deduce the following
$q$-supercongruence: modulo
$\Phi_n(q)(1-aq^n)(a-q^n)(1-bq^n)(b-q^n)(c-q^n)$,
\begin{align}
&\sum_{i_1+i_2+\cdots+i_m\leq n-1}\beta_q(i_1)\cdots\beta_q(i_m)
\notag\\
&\quad\equiv\frac{(1-bq^n)(b-q^n)(c-q^n)\{-1-a^2-a^4+ac+a^3c+a(1+a^2-ac)q^n\}}{(1-ba)(b-a)(1-ac)(a-c)}
\notag\\[1mm]
&\qquad\times[n]^m(b/q^r)^{m(n-r)/d}\frac{(q^{2r}/b;q^d)_{(n-r)/d}^m}{(bq^d;q^d)_{(n-r)/d}^m}
\bigg\{\sum_{k=0}^{(n-r)/d}\frac{(aq^r,q^r/a,q^r/b,cq^{d-r}/b;q^d)_k}{(q^d,q^d/b,q^{2r}/b,cq^d;q^d)_k}q^{dk}\bigg\}^{m}
\notag
\end{align}
\begin{align}
&\quad+\frac{(1-aq^n)(a-q^n)(c-q^n)\{-1-b^2-b^4+bc+b^3c+b(1+b^2-bc)q^n\}}{(1-ab)(a-b)(1-bc)(b-c)}
\notag\\[1mm]
&\qquad\times[n]^m(a/q^r)^{m(n-r)/d}\frac{(q^{2r}/a;q^d)_{(n-r)/d}^m}{(aq^d;q^d)_{(n-r)/d}^m}
\bigg\{\sum_{k=0}^{(n-r)/d}\frac{(bq^r,q^r/b,q^r/a,cq^{d-r}/a;q^d)_k}{(q^d,q^d/a,q^{2r}/a,cq^d;q^d)_k}q^{dk}\bigg\}^{m}.
\notag\\[1mm]
&\quad+\frac{(1-aq^n)(a-q^n)(1-bq^n)(b-q^n)}{(1-ac)(a-c)(1-bc)(b-c)}
\notag\\[1mm]
&\qquad\times[n]^m\frac{(q^r,q^{d-r};q^d)_{(n-r)/d}^m}{(aq^d,q^d/a;q^d)_{(n-r)/d}^m}
\bigg\{\sum_{k=0}^{(n-r)/d}\frac{(q^{d-r},aq^r,q^{r}/a,q^{r}/c;q^d)_k}{(q^d,q^d/b,bq^d,q^{2r}/c;q^d)_k}q^{dk}\bigg\}^{m}.
\label{eq:wei-fff}
\end{align}
Selecting $c=1$ in \eqref{eq:wei-fff}, we gain the following
formula: modulo $\Phi_n(q)^2(1-aq^n)(a-q^n)(1-bq^n)(b-q^n)$,
\begin{align*}
&\sum_{i_1+i_2+\cdots+i_m\leq
n-1}\bar{\beta}_q(i_1)\cdots\bar{\beta}_q(i_m)
\notag\\
&\quad\equiv\frac{(1-bq^n)(b-q^n)(1+a^2-aq^n)}{(1-ba)(b-a)}
\notag\\[1mm]
&\qquad\times[n]^m(b/q^r)^{m(n-r)/d}\frac{(q^{2r}/b;q^d)_{(n-r)/d}^m}{(bq^d;q^d)_{(n-r)/d}^m}
\bigg\{\sum_{k=0}^{(n-r)/d}\frac{(aq^r,q^r/a,q^r/b,q^{d-r}/b;q^d)_k}{(q^d,q^d/b,q^{2r}/b,q^d;q^d)_k}q^{dk}\bigg\}^{m}
\notag
\end{align*}
\begin{align*}
&\quad+\frac{(1-aq^n)(a-q^n)(1+b^2-bq^n)}{(1-ab)(a-b)}
\notag\\[1mm]
&\qquad\times[n]^m(a/q^r)^{m(n-r)/d}\frac{(q^{2r}/a;q^d)_{(n-r)/d}^m}{(aq^d;q^d)_{(n-r)/d}^m}
\bigg\{\sum_{k=0}^{(n-r)/d}\frac{(bq^r,q^r/b,q^r/a,q^{d-r}/a;q^d)_k}{(q^d,q^d/a,q^{2r}/a,q^d;q^d)_k}q^{dk}\bigg\}^{m},
\end{align*}
where $\bar{\beta}_q(k)$ signifies the $c=1$ case of $\beta_q(k)$.
The $b=1$ case of it can be stated as follows: modulo
$\Phi_n(q)^4(1-aq^n)(a-q^n)$,
\begin{align}
&\sum_{i_1+i_2+\cdots+i_m\leq n-1}\lambda_q(i_1)\cdots\lambda_q(i_m)
\notag\\
&\quad\equiv \frac{(1-q^n)^2(1+a^2-aq^n)}{(1-a)^2}[n]^mq^{rm(r-n)/d}
\notag\\
&\qquad\times \frac{(q^{2r};q^d)_{(n-r)/d}^m}{(q^d;q^d)_{(n-r)/d}^m}
\bigg\{\sum_{k=0}^{(n-r)/d}\frac{(q^r,q^{d-r},aq^r,q^r/a;q^d)_k}{(q^d;q^d)_k^3(q^{2r};q^d)_k}q^{dk}\bigg\}^{m}
\notag\\[1mm]
&\quad-\frac{(2-q^n)(1-aq^n)(a-q^n)}{(1-a)^2}[n]^m(a/q^r)^{m(n-r)/d}
\notag\\[1mm]
&\qquad\times\frac{(q^{2r}/a;q^d)_{(n-r)/d}^m}{(aq^d;q^d)_{(n-r)/d}^m}
\bigg\{\sum_{k=0}^{(n-r)/d}\frac{(q^r;q^d)_k^2(q^r/a,q^{d-r}/a;q^d)_k}{(q^d;q^d)_k^2(q^d/a,q^{2r}/a;q^d)_k}q^{dk}\bigg\}^{m}.
\label{eq:wei-ggg}
\end{align}
 Performing the replacements $q\to q^d,a\to
q^{r+n}, b\to q^r/a, c\to q^{d-r}/a, d\to q^d/a, e\to q^d,  f\to
q^{2r}/a, m\to (n-r)/d$ in \eqref{eq:Sear}, we arrive at
\begin{align*}
 _{4}\phi_{3}\!\left[\begin{array}{cccccccc}
q^r/a,& q^{d-r}/a,& q^{r+n}, & q^{r-n} \\
  & q^d, & q^d/a,  & q^{2r}/a
\end{array};q^d,\, q^d
\right]&=a^{(r-n)/d}\frac{(q^{2r},aq^d;q^d)_{(n-r)/d}}{(q^d,q^{2r}/a;q^d)_{(n-r)/2}}
\notag\\
&\quad\times{_{4}\phi_{3}}\!\left[\begin{array}{cccccccc}
q^r,& q^{d-r},& q^{r+n}, & q^{r-n} \\
  & q^{2r}, & aq^d,  & q^d/a
\end{array};q^d,\, q^d
\right].
\end{align*}
Noting that
$(q^{r+n},q^{r-n};q^d)_k\equiv(q^r;q^d)_k^2\pmod{\Phi_n(q)^2}$, we
have

\begin{align}
\sum_{k=0}^{(n-r)/d}\frac{(q^r;q^d)_k^2(q^r/a,q^{d-r}/a;q^d)_k}{(q^d;q^d)_k^2(q^d/a,q^{2r}/a;q^d)_k}q^{dk}
\equiv
a^{(r-n)/d}\frac{(q^{2r},aq^d;q^d)_{(n-r)/d}}{(q^d,q^{2r}/a;q^d)_{(n-r)/2}}
\notag\\[1mm]
\times\sum_{k=0}^{(n-r)/d}\frac{(q^r;q^d)_k^3(q^{d-r};q^d)_k}{(aq^d,q^d/a,q^d,q^{2r};q^d)_k}q^{dk}
\pmod{\Phi_n(q)^2}. \label{eq:wei-hhh}
\end{align}
Substituting \eqref{eq:wei-hhh} into \eqref{eq:wei-ggg}, we are led
to \eqref{eq:wei-aaa} to finish the proof.

\end{proof}

Whereupon, we plan to prove Theorem \ref{thm-c}.

\begin{proof}[Proof of Theorem \ref{thm-c}]
Through \eqref{relation}, we can reformulate \eqref{eq:wei-aaa} as
follows: modulo $\Phi_n(q)^4(1-aq^n)(a-q^n)$,
\begin{align}
&\sum_{i_1+i_2+\cdots+i_m\leq n-1}\lambda_q(i_1)\cdots\lambda_q(i_m)
\notag\\[1mm]
&\qquad\equiv
[n]^mq^{mr(r-n)/d}\frac{(q^{2r};q^d)_{(n-r)/d}^m}{(q^{d};q^d)_{(n-r)/d}^m}
\bigg\{\sum_{k=0}^{(n-r)/d}\frac{(aq^r,q^r/a,q^r,q^{d-r};q^d)_k}{(q^d;q^d)_k^3(q^{2r};q^d)_k}q^{dk}\bigg\}^{m}
\notag\\[1mm]
&\qquad\quad+[n]^mq^{mr(r-n)/d}\frac{(2-q^n)(1-aq^n)(a-q^n)}{(1-a)^2}\frac{(q^{2r};q^d)_{(n-r)/d}^m}{(q^{d};q^d)_{(n-r)/d}^m}
\Omega_n(a), \label{eq:wei-iii}
\end{align}
where
\begin{align*}
\Omega_n(a)&=
\bigg\{\sum_{k=0}^{(n-r)/d}\frac{(aq^r,q^r/a,q^r,q^{d-r};q^d)_k}{(q^d;q^d)_k^3(q^{2r};q^d)_k}q^{dk}\bigg\}^{m}
\notag\\[1mm]
&
\quad-\bigg\{\sum_{k=0}^{(n-r)/d}\frac{(q^r;q^d)_k^3(q^{d-r};q^d)_k}{(aq^d,q^d/a,q^d,q^{2r};q^d)_k}q^{dk}\bigg\}^{m}.
\end{align*}
Thanks to L'H\^{o}pital rule, we derive
\begin{align*}
\lim_{a\to1}\frac{\Omega_n(a)}{(1-a)^2} &=-m
\bigg\{\sum_{k=0}^{(n-r)/d}\frac{(q^r;q^d)_k^3(q^{d-r};q^d)_k}{(q^d;q^d)_k^3(q^{2r};q^d)_k}q^{dk}\bigg\}^{m-1}
\notag\\[1mm]
&\quad\times\sum_{k=0}^{(n-r)/d}\frac{(q^r;q^d)_k^3(q^{d-r};q^d)_k}{(q^d;q^d)_k^3(q^{2r};q^d)_k}q^{dk}
\sum_{t=1}^{k}\bigg\{\frac{q^{dt-d+r}}{(1-q^{dt-d+r})^2}+\frac{q^{dt}}{(1-q^{dt})^2}\bigg\}.
\end{align*}
Letting $a\to1$ in \eqref{eq:wei-iii} and using the above limit, we
discover \eqref{eq:wei-c}.
\end{proof}

\textbf{Acknowledgments}\\

The work is supported by Hainan Provincial Natural Science
Foundation of China (No. 124RC511) and the National Natural Science
Foundation of China (No. 12071103).

\end{document}